\documentclass[10pt, a4paper]{amsart}

\usepackage{a4, a4wide}
\usepackage{amssymb}
\usepackage{amsmath, bm}
\usepackage{amsthm}
\usepackage{amstext}
\usepackage{amscd}
\usepackage{latexsym}
\usepackage{graphics}
\usepackage{color}
\usepackage{enumitem}
\usepackage[all, cmtip]{xy}
\usepackage{enumitem}
\usepackage{comment}
\usepackage[T2A]{fontenc}
\usepackage[main=english,russian]{babel}
\usepackage[colorlinks, pagebackref]{hyperref}
\hypersetup{
  colorlinks=true,
  citecolor=blue,
  linkcolor=blue,
  urlcolor=blue}

\usepackage{cleveref}

\makeatletter
\def\@tocline#1#2#3#4#5#6#7{\relax
  \ifnum #1>\c@tocdepth 
  \else
    \par \addpenalty\@secpenalty\addvspace{#2}%
    \begingroup \hyphenpenalty\@M
    \@ifempty{#4}{%
      \@tempdima\csname r@tocindent\number#1\endcsname\relax
    }{%
      \@tempdima#4\relax
    }%
    \parindent\z@ \leftskip#3\relax \advance\leftskip\@tempdima\relax
    \rightskip\@pnumwidth plus4em \parfillskip-\@pnumwidth
    #5\leavevmode\hskip-\@tempdima
      \ifcase #1
      \or\or \hskip 2em \or \hskip 2homologyem \else \hskip 3em \fi%
      #6\nobreak\relax
    \dotfill\hbox to\@pnumwidth{\@tocpagenum{#7}}\par
    \nobreak
    \endgroup
  \fi}
\makeatother

\theoremstyle{plain}

\newtheorem{theorem}{Theorem}[section]

\newtheorem{lemma}[theorem]{Lemma}
\newtheorem{corollary}[theorem]{Corollary}
\newtheorem{proposition}[theorem]{Proposition}

\theoremstyle{definition}

\newtheorem{remark}[theorem]{Remark}

\newtheorem{definition}[theorem]{Definition}

\numberwithin{equation}{section}

\newcommand{\Sm}{{\rm Sm}}

\newcommand{\rank}{{\rm rank}}

\newcommand{\Pic}{{\rm Pic}}

\newcommand{\Hom}{{\rm Hom}}

\newcommand{\Frac}{{\rm Frac}}

\newcommand{\Spec}{{\rm Spec \,}}

\newcommand{\ie}{{\it i.e.\/},\ }
\renewcommand{\iff}{\mbox{ $\Longleftrightarrow$ }}

\newcommand{\nis}{Nis}

\newcommand{\piA}{{\bm \pi}^{\A^1}}   		

\newcommand{\A}{\mathbb A}

\newcommand{\simpp}{{\rm{\Delta^{op}Pshv}}}

\def\<{\langle}
\def\>{\rangle} 
\def\-{\overline} 
\def\~{\widetilde}
\def\^{\widehat}

\def\@{\mathcal}
\def\#{\mathbb}
\def\&{\mathbf}
\def\_{\underline}
 
\def\x{\times}

\input{xy}
\xyoption{all}

\begin{document}

\title[$\#A^1$-connectedness of moduli of semistable bundles and symplectic bundles on a curve]{$\#A^1$-connectedness of moduli of semistable bundles and symplectic bundles on a curve}

\author[U. V Dubey]{Umesh ~V ~Dubey}
\address{Harish-Chandra Research Institute,
Prayagraj - 211019, Uttar Pradesh, India} 
\address{Homi Bhabha National Institute, Training School Complex, Anushakti Nagar, Mumbai 400 094, India}
 \email{umeshdubey@hri.res.in}
\author[R. Pawar]{Rakesh Pawar}
	\address{School of Mathematical Sciences, National Institute of Science Education and Research,
		Bhubaneswar, Odisha- 752050, India.}
        \address{Homi Bhabha National Institute, Training School Complex, Anushakti Nagar, Mumbai 400 094, India}
	\email{rakeshpawar@niser.ac.in}
	

\subjclass[2020]{14D20, 14D23, 14F42}
\keywords{Moduli stacks of vector bundles, symplectic bundles, $\#A^1$-connectedness}
\begin{abstract} 
In this note,  we show that over a geometrically irreducible smooth projective curve over an infinite field $k$, the  moduli stack of semistable vector bundles of fixed determinant is $\mathbb{A}^1$-connected if and only if the moduli stack admits a $k$-rational point. For this we use Langton's elementary modifications for vector bundles at the level of families. 
In addition, we prove the $\mathbb{A}^1$-connectedness of the moduli stack of symplectic bundles with forms valued in a fixed line bundle $L$ on a smooth projective curve $C$ of genus $g \ge 2$ over an infinite field $k$ with $C(k)\neq \emptyset$.  As an application, we deduce the $\mathbb{A}^1$-connectedness of moduli stack of quasi-parabolic symplectic vector bundles with forms valued in a fixed line bundle.
\end{abstract}

\maketitle
\tableofcontents

\section{Introduction}

Morel-Voevodsky developed in~\cite{MV99} the framework of homotopy theory suitable for schemes, called $\#A^1$-homotopy theory. Ever since then, there has been a significant amount of work in understanding $\A^1$-homotopy types of various algebraic geometric objects. In particular, one is interested in the study of $\A^1$-homotopy type of certain moduli stacks. 

To any stack one can associate a simplicial
sheaf via the nerve construction (see \cite{H08}) and hence defines an object in the $\#A^1$-homotopy
category. In this direction, in~\cite{HY24} it was shown that the moduli stack of vector bundles of fixed determinant on a smooth projective curve over an infinite field is $\#A^1$-connected. Further, the $\#A^1$-connectedness of the moduli stack of semistable (quasi-/parabolic) vector bundles of fixed determinant on a smooth projective curve over an algebraically closed field is considered in~\cite{CC25}.

In this paper, we study the $\#A^1$-connectedness of certain moduli stacks on a smooth projective curve over a field. 
Let $C$ be a smooth projective curve over a field $k$. Let $L$ be line bundle on $C$ and $n\geq 1$ be an integer. 

\begin{enumerate}
    \item In $\S 3$ we study $\#A^1$-connectedness of the moduli stack $Bun_{n, L}^{ss}$ semistable vector bundles of rank $n$ on a curve $C$  of fixed determinant $L$ and propose an alternative proof than in~\cite{CC25} which works over any infinite base field $k$.
    \item In $\S 4$ we study the $\#A^1$-connectedness of the moduli stack $Sympl_{2n, L}$ of rank $2n$ symplectic vector bundles  on $C$ with symplectic form valued in a fixed line bundle $L$ over the base curve $C$.
    \item As a consequence of above, in $\S 5$ we study the  $\#A^1$-connectedness of the moduli stack $qPSympl^{e, \_m}_{2n, L(D)}$ of quasi-parabolic symplectic bundle of rank $2n$ with quasi-parabolic data $(e, \_m)$ on $C$ of rank $2n$ along a finite subset $D\subset C$ of distinct $k$-rational points. 
    
\end{enumerate}

More precisely, we prove the following results:
   \begin{theorem}[\Cref{thm:main-text-ss}]
   Let $C$ be a geometrically irreducible smooth projective curve over an infinite field $k$. Then the moduli stack 
$Bun^{ss}_{n, L}$ of rank $n$ semistable vector bundles of fixed determinant $L$ on $C$ is $\#A^1$-connected iff it admits a $k$-rational point.
    
\end{theorem}
\begin{corollary}[Corollary~\ref{cor:genusatleast2}]
    
   Let $C$ be a geometrically irreducible smooth projective curve of genus $\geq 2$ over an infinite field $k$ with $C(k)\neq \emptyset$. Let $L\in \Pic(C)$ and $n\geq 1$. Then 
    the moduli stack $Bun^{ss}_{n, L}$ of rank $n$ semistable vector bundles of fixed determinant $L$ on $C$ admits a $k$-rational point and hence is $\#A^1$-connected. 
\end{corollary}

\begin{theorem}[\Cref{thm:main-text}]\label{thm:main}
   Let $C$ be a geometrically irreducible smooth projective curve over a field $k$ of genus $g\geq 2$ and $C$ admits a $k$-rational point. Let $L$ be line bundle on $C$ and $n\geq 1$. Then 
    the moduli stack $Sympl_{2n, L}$ of rank $2n$ symplectic vector bundles  on $C$ is $\#A^1$-connected. 
\end{theorem}

Fix a line bundle $L$ on $C$ over $k$ and a quasi-parabolic data $(e, \_m)$ on $C$ of rank $2n$ along a divisor $D$ consisting of $k$-rational points in the sense of~\Cref{def:quasi-para-data}. One endows $L(D)=L\otimes_{\@O_C} \@O_C(D)$ with a quasi-parabolic bundle structure where for each point $p\in D,$ with a filtration $\@L_p\supset 0$ as in \cite[Remark 2.2]{BCD26}. Following the works of~\cite{BMW11, BCD26}, we consider the moduli stack $qPSympl^{e, \_m}_{2n, L(D)}$ of quasi-parabolic symplectic bundles of rank $2n$ on $C$ and show the following
\begin{theorem}[see~\Cref{thm:quasi-parabolic-sympl}]\label{intro-thm:quasi-parabolic-sympl}
        Let $C$ be a geometrically irreducible smooth projective curve of genus $g\geq 2$ over a field $k$ and a finite subset $D \subset C$ of distinct $k$-rational points. Fix a quasi-parabolic data $(e, \_m)$ on $C$ of rank $2n$ along $D$ in the sense of~\Cref{def:quasi-para-data}. 
    
    Then $qPSympl^{e, \_m}_{2n, L}$ is $\#A^1$-connected. 
\end{theorem}

\paragraph{\bf Structure of the paper:}
In section~\ref{sec:prelim}, we recall preliminaries on $\#A^1$-homotopy theory and some useful results for the main paper. In section~\ref{sec:semistable}, we study $\#A^1$-connectedness pf the moduli stack of semistable vector bundles. 
In section~\ref{sec:symplectic}, we study $\#A^1$-connectedness of the moduli stack of symplectic vector bundles.  In section~\ref{sec:quasi-pb-symplectic}, we study $\#A^1$-connectedness of the moduli stack of quasi-parabolic symplectic vector bundles as an application.

\subsection*{Acknowledgments}
The first name author is partially supported by ANRF under project no. ANRF/ARGM/2025/002470/MTR and HRI, Prayagraj.
The second-named author is supported by NISER, HBNI, DAE, India.

\section{Preliminaries}\label{sec:prelim}

\subsection{\texorpdfstring{$\#A^1$-homotopy theory}{hom}}
First, we recall some basic notions in the homotopy theory.

Let $k$ be a field. 
$\Sm/k$ denotes the category of finite type smooth $k$-schemes. Let $(\Sm/k)_{\nis}$ denote the site with the Nisnevich topology on the category $\Sm/k$. The category of simplicial presheaves $Spc_k:=\simpp(\Sm/k)_{\nis}$ is a simplicial model category with the local-Ninsecich-model structure as in~\cite[Theorem 1.4]{MV99}. By \cite[section 3.2, page 105]{MV99}, the above simplicial model category   $Spc^{\A^1}_k:=\simpp(\Sm/k)_{\nis}$ is endowed with the $\A^1-$model category structure. An object of $Spc^{\A^1}_k$ is simply referred to as a \textit{Space}. Inverting the $\A^1-$weak equivalences in $Spc^{\A^1}_k$, the resulting homotopy category is the homotopy category of smooth schemes over $k$ denoted by $\@H(k).$ 

 For any object $\@X, \@Y\in Spc^{\A^1}_k$, we denote $[\@X, \@Y]_{\A^1}:=\Hom_{\@H(k)}(\@X, \@Y)$. 
 \begin{definition}\label{def:A1-conn}
    Let $\@X$ be a simplicial presheaf on $\Sm/k$. Define $\piA_0(\@X)$, \emph{the sheaf of $\#A^1$-connected components of $\@X$}, as the Nisnevich sheafification of the presheaf  on $\Sm/k$ $$U\mapsto [U, \@X]_{\A^1}$$ 
    for each $U\in \Sm/k$. 

    We say $\@X$ is $\#A^1$-connected if the Nisnevich  sheaf $\piA_0(\@X)=*.$
 \end{definition}

\begin{definition}\label{def:naive-A1}
    Let $\@X$ be a simplicial presheaf on $\Sm/k$. Let $U$ be a $k$-scheme. Let $f_0, f_1\in \@X(U)_0$ be $0$-simplices of $\@X(U)$. We say $f_0, f_1$ are naively $\#A^1$-homotopic if there is $h: \#A^1_U\to \@X$ such that $h(0)=f_0$ and $h(1)=f_1.$ 
\end{definition}
\begin{definition}\label{def:naive-chain-A1}
    Let $\@X$ be a simplicial presheaf on $\Sm/k$. Define $\@S(\@X)$, \emph{the sheaf of naive $\#A^1$-chain connected components of $\@X$}, as the Nisnevich sheafification of the presheaf  on $\Sm/k$ assigning $$U\mapsto {\@X(U)_0}/\sim$$ for each $U$ in $\Sm/k$ where $\sim$ is the equivalence relation generated by naive $\#A^1$-homotopies. 

    We say $\@X$ is naively $\#A^1$-chain connected if for each finitely generated field extension $K$ of $k$, $\@S(\@X)(K)=*$.
\end{definition}

We recall the following lemma which will be used in the main result.
\begin{lemma}(\cite[Section 2,
Corollary 3.22]{MV99} with~\cite[Lemma 6.1.3.]{Mor05})\label{lem:a1-chain-a1-conn}
    Let $\@X$ be a simplicial presheaf on $\Sm/k.$ If $\@X$ is naively $\#A^1$-chain connected \ie $\@S(\@X)(K)=*$ for each finitely generated field extension $K$ of $k$, then $\@X$ is $\#A^1$-connected. 
\end{lemma}

\subsection{Space associated to a stack}
For any stack $\@Y$ on a field $k$, using the nerve construction (see \cite{H08}) one can associate a simplicial nerve presheaf $N(\@Y)$ on $\Sm/k$, where for each $k$-scheme $U$, the 0-simplices on $U$ are given by objects of the groupoid $\@Y(U)$.  

In the view of the \Cref{def:naive-A1}, we say two objects $f_0, f_1$ of the groupoid $\@Y(U)$ on a scheme $U$, are   naively $\#A^1$-homotopic if there is an object $h$ of the groupoid $\@Y(\#A^1_U)$ such that $h(0)=f_0$ and $h(1)=f_1.$ 

Similarly, given a stack $\@Y$, one can define $\@S(\@Y)$, \emph{the sheaf of naive $\#A^1$-chain connected components of $\@Y$} as \emph{the sheaf of naive $\#A^1$-chain connected components} of the simplicial presheaf $N(\@Y)$.    

We say a stack $\@Y$ is $\#A^1$-connected if associated similicial nerve presheaf $N(\@Y)$ is $\#A^1$-connected in the sense of~\Cref{def:A1-conn}. Similarly we say a stack $\@Y$ is naively $\#A^1$-chain connected if associated similicial nerve presheaf $N(\@Y)$ is $\#A^1$-chain connected in the sense of~\Cref{def:naive-chain-A1}.

\section{\texorpdfstring{$\#A^1$-connectedness of moduli stack of semistable vector bundles on curve}{}}\label{sec:semistable}

\emph{Let $C$ be a geoemtrically irreducible smooth projective curve over an infinite field $k$ such that $C(k)\neq\emptyset$. Let $L$ be a line bundle on $C$ and $n\geq 1$ be an integer.}

 In this section,  we consider the moduli stack of vector bundles of rank $n$ of fixed determinant $L$ on curve. Let us denote by 
$Bun_{n, L}$ the moduli stack defined by groupoids for each $k$-scheme $S$ 
\begin{itemize}
    \item the objects are locally free sheaves $\@E$ of rank $n$ on $C\x S$ paramterized by $S$ and
    \item morphisms as isomorphism of locally free sheaves on $C\x S.$ 
\end{itemize}
If the base field $k$ is infinite, in \cite{HY24} it is shown that $Bun_{n, L}$ is $\#A^1$-connected. 

We consider the moduli stack of semistable vector bundles of rank $n$ of fixed determinant $L$ on curve. More precisely, 
$Bun^{ss}_{n, L}$ is the open substack of $Bun_{n, L}$  defined by groupoids for each $k$-scheme $S$ 
\begin{itemize}
    \item the objects are semistable locally free sheaves $\@E$ of rank $n$ on $C\x S$ parametrized by $S$.
\end{itemize}

    In this section, we study the $\#A^1$-connectedness of the moduli stack $Bun^{ss}_{n, L}$.  We provide an alternative argument than in~\cite{CC25}:

First we note following two lemmas which are of independence interest. We believe the results maybe be known to experts and the main idea is from~\cite{Langton75}, but due to lack of reference in the literature we present them below.  

\subsection{Langton type modification of families}

Let $K$ be a finitely generated field extension of $k.$ By base change via  $\Spec K\to \Spec k$, by replacing $C$ by $C\x_k \Spec K$, (for the ease of notation) we can assume that $C$ is defined over the field $K.$
Let $i_0: C=C\x \{0\}\to C\x \#A^1_K$ and  $i_1: C=C\x \{1\}\to C\x_k \#A^1_K$ be the inclusions. 

Let $\@E$ be a vector bundle on $C\x_k \#A^1_K$. By~\cite[Proposition 2.11]{ABBLT}  there is a Zariski open subset $U\subset \#A^1_K$ such that for each $u\in U$, the restriction $\@E|_{C\x_k \Spec K(u)}$ is semistable on $C\x \Spec K(u)$. Let us refer to maximal open such $U\subset \#A^1_K$ as the semistable locus of $\@E$.

In the following lemma, we modify a single family of generically semistable vector bundles parametrized by $\#A^1$ into a fully semistable family. 
\begin{lemma} \label{lem:modify-1-family} 

Let $\@E$ be a vector bundle on $C\x \#A^1_K$  such that $i_0^*\@E$ is semistable on $C$ via  $i_0: C=C\x \{0\}\to C\x \#A^1_K$. Let $U\subset \#A^1_K$ be the maximal semistable locus of $\@E$ in $C$. Assume that $1\notin U$. Then there exists a vector bundle $\~{\@E}$ on  $C\x \#A^1_K$  such that 
\begin{enumerate}
\item $\~{\@E}$ is semistable on  $C\x (\#A^1_K-\{1\})$ 
\item $\~{\@E}|_{C\x \{1\}}= \@E|_{C\x \{1\}}$ and
    \item $\~{\@E}|_{C\x U}= \@E|_{C\x U}.$    
\end{enumerate}
\end{lemma}

\begin{proof}
    Let $U\subset \#A^1_K$ be the semistable locus of $\@E$, so that by assumption, $0\in U$ and $\@E|_{C\x U}$ is semistable family. 
    This corresponds to $h: U\to Bun^{ss}_{n, L}$. By \cite[Proposition 6]{Langton75} (see also \cite[Theorem 3.7]{ABBLT}), the moduli stack $Bun^{ss}_{n, L}$ is universally closed, hence the morphism $h$ extends to a morphism (not-necessarily unique) $\-h: \#A^1_K-\{1\}\to Bun^{ss}_{n, L}$.     The morphism    $\-h$ corresponds to semistable vector bundle $\~{\@E}$ on  $C\x (\#A^1_K-\{1\})$  such that there is an isomorphism $$\phi: \~{\@E}|_{C\x U}\simeq  \@E|_{C\x U}.$$ 
We glue $(\@E|_{C\x (U\cup \{1\})}, C\x (U\cup \{1\}))$ and $(\~{\@E}, C\x (\#A^1_K-\{1\}))$ along the intersection $$(C\x (U\cup \{1\})) \cap (C\x (\#A^1_K-\{1\}))= C\x U $$ via the isomorphism $\phi$, to get the desired vector bundle $\~{\@E}$ on  $C\x \#A^1_K$.
\end{proof}

In the next lemma, given two  families of generically semistable vector bundles parametrized by $\#A^1$ which intersect into an unstable vector bundle, we modify into a family with semistable intersection. 

\begin{lemma} \label{lem:modif-2-family}

Let $\mathcal E$ and $\mathcal F$ be two  vector bundles on $C\x \#A^1_K$  where $U\subset \#A^1_K$ is the semistable locus of $\@E$  such that 
\[
\-E=i_1^{*}(\mathcal E) = i_0^{*}(\mathcal F)=\-F.
\] 
Assume that $0\in U$ and $1\notin U$. 
Then there is a semistable bundle $\~{\@E}$ and a bundle (not-necessarily semistable) $\~{\@F}$ on $C\x \#A^1_K$ such that   

\begin{enumerate}
\item $i_1^{*}(\~{\@E}) = i_0^{*}(\~{\@F}).$ 

    \item $\~{\@E}|_{C\x U}= \@E|_{C\x U}$ 

\item $\~{\@F}|_{C\x \{1\}}= \@F|_{C\x \{1\}}$ 

\end{enumerate}
\end{lemma}

\begin{proof}
We prove this in a sequence of steps. 

\medskip
\noindent
{\bf Step 1 - Elementary modification of first family :} 
Applying ~\Cref{lem:modify-1-family} to $\@E$ we get $\~{\@E}$ such that its semistable locus is equal to $\#A^1-\{1\}$ satisfying   
\begin{enumerate}
\item $\-E=i_1^{*}(\~{\@E}) = i_0^{*}(\mathcal F)=\-F$
\item $\~{\@E}$ is semistable on  $C\x (\#A^1_K-\{1\})$ 
    \item $\~{\@E}|_{C\x U}= \@E|_{C\x U}.$ 
    
\end{enumerate}
For the rest of the proof, it is enough to replace $\@E$ by $\~{\@E}$ as above.

By our assumption $\-E$ is unstable. Let $R_1$ be the local ring of $\#A^1_K$ at 1, which will be used to modify $\@E$, and let us denote by $R_0$ as the local ring of $\#A^1_K$ at 0 for $\@F$, hence $R$ (= either $R_0$ or $R_1$) is a DVR.

Let us denote by $E=i^*\@E$ for $i: C\x \Spec R\to C\x \#A^1_K$ and $F=i^*\@F$ for $i: C\x \Spec R\to C\x \#A^1_K$. We denote by $\-E$ (and $\-F$) the restriction of $E$ (and $F$) to the closed fiber $C\x \Spec K\hookrightarrow C\x \Spec R_1$ (and $C\x \Spec K\hookrightarrow C\x \Spec R_0$ respectively.)  
Note that by our assumption $\-E=\-F.$

\medskip
\noindent
{\bf Step 2 - Simultaneous elementary modification : }\\
By assumption,  $\-E$ is not semistable. Following the notation from~\cite{Langton75},   let $\-B$ be the $\beta$-subbundle of $\-E.$ We consider the exact sequence $$0\to \-B\to \-E\to \-Q\to 0$$  
Let $E^{(1)}=\ker (E\to \-E\to \-Q)$. Let us denote by $F^{(1)}=\ker (F\to \-F=\-E\to \-Q)$. 

Note that $E^{(1)}$ (and $F^{(1)}$) is a locally free sheaf on $C_{R_1}$  (and on $C_{R_0}$ resp.) such that the  restriction of $E^{(1)}$ (and $F^{(1)}$) to the closed fiber $C\x \Spec K\hookrightarrow C\x \Spec R_1$ (and $C\x \Spec K\hookrightarrow C\x \Spec R_0$ respectively) satisfies  $\-{E^{(1)}}=\-{F^{(1)}}.$ 

Now let  $\-{B^{(1)}}$ be the $\beta$-subbundle of $\-{E^{(1)}}.$ We consider the exact sequence $$0\to \-{B^{(1)}}\to\-{E^{(1)}}\to \-{Q^{(1)}}\to 0$$  
 We define $E^{(2)}=\ker (E^{(1)}\to \-{E^{(1)}}\to \-{Q^{(1)}})$ and  $F^{(2)}=\ker (F^{(1)}\to \-{F^{(1)}}=\-{E^{(1)}}\to \-{Q^{(1)}}.)$
As before note that $E^{(2)}$ (and $F^{(2)}$) is a locally free sheaf on $C_{R_1}$  (and on $C_{R_0}$ resp.) such that the  restriction of $E^{(1)}$ (and $F^{(1)}$) to the closed fiber $C\x \Spec K\hookrightarrow C\x \Spec R_1$ (and $C\x \Spec K\hookrightarrow C\x \Spec R_0$ respectively) satisfies    $\-{E^{(2)}}=\-{F^{(2)}}.$
 
Continuing this way for each $m\geq 1$, we get a sequence $\{E^{(m)}\}$ (and $\{F^{(m)}\}$) of locally free sheaves on $C_{R_1}$  (and on $C_{R_0}$, resp.) such that the  restriction of $E^{(m)}$ (and $F^{(m)}$, resp.) to the closed fiber $C\x \Spec K\hookrightarrow C\x \Spec R_1$ (and $C\x \Spec K\hookrightarrow C\x \Spec R_0$ respectively) satisfies    $\-{E^{(m)}}=\-{F^{(m)}}.$ 

By \cite{Langton75} for $m$ large enough,  $\{E^{(m)}\}$ extends $E_{\Frac R}$ on $C\x \Spec \Frac R$ and its restriction to the closed fiber $\-{E^{(m)}} $ ( $=\-{F^{(m)}}$) is semistable.  

We fix such $m$. We can extend $E^{(m)}$ on $C\x \Spec R_1$ to $\@E_1$ on $C\x U_1$ for a Zariski open neighbourhood $U_1$ of 1 and also  $\@F_0$ on $C\x V_0$ for a Zariski open neighbourhood $V_0$ of 0 such that \begin{enumerate}
    \item $\@E_1|_{C\x \Spec R_1}= E^{(m)}|_{C\x \Spec R_1}$ 
    \item $\@E_1|_{C\x (U_1-\{1\})}\simeq  \@E|_{C\x (U_1-\{1\})}$ 
    \item $\@F_0|_{C\x \Spec R_0}= F^{(m)}|_{C\x \Spec R_0}$ 
    \item $\@F_0|_{C\x (V_0-\{0\})}\simeq  \@E|_{C\x (V_0-\{0\})}$  and
    \item $\@E_1|_{C\x \{1\}}=\@F_0|_{C\x \{0\}}.$
\end{enumerate}

\noindent
{\bf Step 3 - Construction of $\~{\@E}$ :}\\
Let us write $U_0=\#A^1_K-\{1\}$. Then $\{U_0, U_1\}$ is a Zariski open cover of $C\x \#A^1_K$ such that $U_0\cap U_1=U_1-\{1\}$. Via the isomorphism from (2), $$(\@E_1)|_{C\x (U_1\cap U_0)}\simeq (\@E|_{C\x U_0})|_{C\x (U_0\cap U_1)}$$ induced by $$\@E_1|_{C\x (U_1-\{1\})}\simeq  \@E|_{C\x (U_1-\{1\})}$$ we glue $(U_0, \@E|_{U_0})$ and $(U_1, \@E_1)$  to get a locally free sheaf $\~{\@E}$ on $C\x \#A^1_K$ with desired properties. 

\medskip 
\noindent 
{\bf Step 4 - Construction of $\~{\@F}$ :}\\ 
Similarly we construct $\~{\@F}$. Let us write $V_1=\#A^1_K-\{0\}$. Then $\{V_0, V_1\}$ is a Zariski open cover of $C\x \#A^1_K$ such that $V_0\cap V_1=V_0-\{0\}$. Via the isomorphism from (4), $$(\@F_0)|_{C\x (V_0\cap V_1)}\simeq (\@F|_{C\x V_1})|_{C\x (V_1\cap V_0)}$$ induced by $$\@F_0|_{C\x (V_0-\{0\})}\simeq  \@F|_{C\x (V_0-\{0\})}$$ we glue $(V_1, \@E|_{V_1})$ and $(V_0, \@F_0)$  to get a locally free sheaf $\~{\@F}$ on $C\x \#A^1_K$ with desired properties. 
      
\end{proof}

\subsection{Moduli stack of semistable bundles }
We note the following consequence of this observation, which gives an alternate proof of $\#A^1$-connectedness of the moduli stack of semistable bundles on a curve. 

\begin{theorem}\label{thm:main-text-ss} Let $C$ be a geometrically irreducible smooth projective curve over an infinite field $k$. Then the moduli stack 
$Bun^{ss}_{n, L}$ of rank $n$ semistable vector bundles of fixed determinant $L$ on $C$ is $\#A^1$-connected iff it admits a $k$-rational point.
    
\end{theorem}

\begin{proof}
 Assume $Bun^{ss}_{n, L}$ is $\#{A}^1$-connected. Then by \cite[Remark 2.5, page 106]{MV99}, the moduli stack $Bun^{ss}_{n, L}$  admits a $K$-rational point for any finitely generated field extension $K$ over $k$ (for this we note that for such $K$, $\Spec K$ is a Henselian local scheme).

    For the converse, we note that as $Bun^{ss}_{n, L}(k) \neq \emptyset$, then $Bun^{ss}_{n, L}(K) \neq \emptyset$ for any finitely generated field extension $K$ over $k$. 
    By~\Cref{lem:a1-chain-a1-conn}, it is enough show the following: let $K$ be a finitely generated field extension of $k.$ 
Let $E_0$ and $E_1$ be two semistable bundles on $C\x \#A^1_K$. Then there exists a chain of semistable bundles parametrized by $\mathbb A^1_K$ connecting $E_0$ and $E_1$. 

As shown in~\cite{HY24}, $E_0$ and $E_1$ are connected by a chain of bundles, say $\@E_1, \dots, \@E_r$  parametrized by $\#A^1_K$. We show that there exists a chain of semistable bundles parametrized by $\mathbb A^1_K$ connecting $E_0$ and $E_1$. We show this by induction on $r,$ the length of the $\#A^1$-chain. 

    Assume $r=1$. Let $U\subset \#A^1_K$ be the semistable locus of $\@E$ so that $0, 1\in U$ and $\@E|_{C\x U}$ is semistable family. 
    This corresponds to $h: U\to Bun^{ss}_{n, L}$. By \cite[Proposition 6]{Langton75} (see also \cite[Theorem 3.7]{ABBLT}), the moduli stack $Bun^{ss}_{n, L}$ is universally closed, hence the morphism $h$ extends to a morphism (not-necessarily unique) $\-h: \#A^1_K\to Bun^{ss}_{n, L}$.     The morphism    $\-h$ corresponds to semistable vector bundle $\~{\@E}$ on  $C\x \#A^1_K$  such that $$\~{\@E}|_{C\x U}= \@E|_{C\x U}$$ and 
    $$\~{\@E}|_{C\x 0}= E_0 \ \ \text{and} \ \ \~{\@E}|_{C\x 1}= E_1.$$

    Assume $r>1$ and by the induction hypothesis, any chain of length $r-1$ connecting any two semistable bundles on $C\x \Spec K$ can be joined by an $\#A^1$-chain of semistable bundles parametrized by $\#A^1_K$.
        
        Consider the first two intersecting families $\@E_1$ and $\@E_2.$  There are two possibilities: they intersect in a semistable bundle or an unstable bundle. In the case that these intersect in a semistable bundle, then applying~\Cref{lem:modify-1-family} to the first chain and applying the induction hypothesis to the remaining chain proves the claim.

        In the case these intersect in an unstable bundle, applying~\Cref{lem:modif-2-family} to $\@E_1$ and $\@E_2$, we get semistable bundle $\~{\@E_1}$ and not-necessarily semi-stable bundle $\~{\@E_2}$ which intersect in a semi-stable bundle say $F$. Then $F$ is connected to $E_2$ by a chain of bundles of length $r-1$, hence by induction $F$ is connected to $E_2$ by a chain of semistable bundles. Hence we are done in this case as well.   
\end{proof}
In the case the genus $g\geq 2$, we can make the following uniform assertion.
\begin{corollary}\label{cor:genusatleast2}
     Let 
   Let $C$ be a geometrically irreducible smooth projective curve of genus $\geq 2$ over an infinite field $k$ with $C(k)\neq \emptyset$. Let $L\in \Pic(C)$ and $n\geq 1$. Then 
    the moduli stack $Bun^{ss}_{n, L}$ of rank $n$ semistable vector bundles of fixed determinant $L$ on $C$ admits a $k$-rational point and hence is $\#A^1$-connected. 
\end{corollary}
\begin{proof} Since $g\ge 2$ and $k$ is infinite, by \cite[Lemma 1.5]{Hof07}, the open substack $Bun^{ss}_{n, L}$ admits a $k$-rational point. The corollary follows from~\Cref{thm:main-text-ss}.
\end{proof}
\begin{remark}
    In the case the genus is 0, the moduli stack $Bun^{ss}_{n, L}$ may or may not admit a $k$-rational point, for example $Bun^{ss}_{n, L, \#P^1}(k)$ is non-empty iff $n\mid \deg L$. Also by~\cite{Tu93}, in case genus of $C$ is 1, and $k=\-k$, $Bun^{ss}_{n, L, C}(k)\neq \emptyset.$ 
\end{remark}
\section{\texorpdfstring{$\#A^1$-connectedness of moduli stack of symplectic bundles}{}}\label{sec:symplectic}

In this section, we will prove the $\#A^1$-connectedness of the moduli stack of symplectic bundles. 
For the rest of the paper, we will follow the following assumptions, unless mentioned otherwise. 

Let $C$ be a geometrically irreducible smooth projective curve over a field $k$ of genus $g\geq 2$ and $C(k)\neq \emptyset.$ Let $L$ and $I$ be line bundles on $C$ and $n\geq 1$ be an integer.

We consider the moduli stack $Sympl_{2n, L}$ of rank $2n$ symplectic vector bundles  on $C$ given by groupoids  $Sympl_{2n, L}(S)$ on each $k$-scheme $S$ as follows: 
\begin{itemize}
    \item Each object consists of a tuple $$(\@E, b: \@E\otimes \@E\to \@L)$$ where $\@E$ is a vector bundle of rank $2n$, $\@L$ a line bundle on $C\x_k S$, Zariski locally on $S$ isomorphic to $\pi^*L$ where $\pi:C\x_k S\to C$ is the projection and a non-degenerate symplectic bilinear form $b: \@E\otimes \@E\to \@L.$
    \item A morphism $(\@E, b: \@E\otimes \@E\to \@L)\to (\@E', b': \@E'\otimes \@E'\to \@L')$ consists of isomorphism of vector bundles $\@E\to \@E'$ and $\@L\to \@L'$ which are compatible with $b$ and $b'.$
\end{itemize} 
This stack is algebraic (in the sense of Artin) and locally of finite type over $k$. By standard
deformation theory, it is smooth of dimension $n(2n + 1)(g-1)-1$ over $k$. Moreover it is shown by Biswas-Hoffman (~\cite[Lemma 5.2]{BH08} that the moduli stack $Sympl_{2n, L}$ is birational to $B\#G_m \times \#A^s$
 for some $s$, if $n$ and the degree of $L$ are both odd. 
 
  We recall from \cite[Definition 3.2]{BH08} that $$\Phi_I: \@Sub_{1, I, 2n}(\@E^{univ}) \to Sympl_{2n, L}$$ is the moduli stack of rank $2n$ symplectic vector bundles together with a sub-bundle of $E$ isomorphic to $I.$ More precisely, consider the moduli stack  $\@Sub_{1, I, 2n}(\@E^{univ}) $  given by groupoid on each $k$-scheme $S$ by 

\begin{itemize}
    \item Each object consists of a tuple $$(\@E, b: \@E\otimes \@E\to \@L, \@I\subset \@E)$$ in $Sympl_{2n, L}(S) $ and a subbundle $\@I\subset \@E$ such that $\@I$ is (Zariski) locally in $S$ isomorphic to the pullback $\pi^*I$ where $\pi: C\x_k S\to C$ is the projection morphism. 
    
    \item A morphism $$(\@E, b: \@E\otimes \@E\to \@L, \@I\subset \@E)\to (\@E', b': \@E'\otimes \@E'\to \@L', \@I'\subset \@E')$$ consists of isomorphism of vector bundles $$\@E\to \@E', \@L\to \@L' \ \text{and} \ \@I\to \@I'$$ which are compatible with $b$, $b'$ and the inclusions $\@I\subset \@E$ and $\@I'\subset \@E'$.
\end{itemize} 
We also recall from \cite[Definition 3.4]{BH08} that there is a morphism of moduli stacks $$ \Phi_I': {\bf \@{E}xt}(\@F^{univ}, I, 2n-2) \to \@Sympl_{2n-2, L}$$ 
sending $$(\@F, \ c: \@F\otimes \@F\to \@L, \ 0\to \@I\to \@H\to \@F\to 0)\mapsto (\@F, \ c: \@F\otimes \@F\to \@L)$$ with the morphisms described above. 

There is also a morphism $\Psi_I$ connecting the above two morphisms \cite[Definition 3.6]{BH08} defined by sending
$$
(\@E, b: \@E\otimes \@E\to \@L, \@I\subset \@E) \mapsto (\@F:= \@I^\perp / \@I, \ c: \@F\otimes \@F\to \@L, \ 0\to \@I\to \@H := \@I^\perp \to \@F \to 0)
$$
where $c$ is the symplectic form on $\@F$ induced by $b$ on $\@E$.

To summarize, we have a diagram of moduli stacks
\begin{equation}\label{eq:aux-stacks}
    \xymatrix{
\@Sub_{1, I, 2n}(\@E^{univ}) \ar[r]^-{\Psi_I} \ar[d]_{\Phi_I} & \@Ext(\@F^{univ}, I, 2n-2) \ar[d]^{\Phi'_I} \\
Sympl_{2n, L}  & Sympl_{2n-2, L}
}
\end{equation}
We exploit this diagram to reduce the $\#A^1$-connectedness of $Sympl_{2n, L}$ to the $\#A^1$-connectedness of the other auxiliary moduli stacks appearing in the diagram.

We first show the following proposition, which will be used in the proof of $\#A^1$-connectedness of $Sympl_{2n, L}$ 
\begin{proposition}\label{prop:ext}
Let $C$ be a geometrically irreducible smooth projective curve over a field $k$ of genus $g\geq 2$ and let $K$ be a finitely generated field extension of $k$. Let $\@Q, \@K, \~{\@K}$ be vector bundles on $C\x_k \#A^1_K$. Let $f: \~{\@K}\to \@K$ be a surjective morphism of vector bundles. Assume the diagram  $$\xymatrix{
   \@M \ar[r]^-{\phi} \ar[d]^{\gamma} & \@Ext(\@Q, \~{\@K})  \ar[d]^{f_*} \\ 
   \#A^1_K\ar[r] &  \@Ext(\@Q, \@K) 
   }$$  is a cartesian diagram of stacks over the field $k$. Then the following holds:
   \noindent\begin{enumerate}
       \item $\@Ext(\@Q, \~{\@K})$ and $\@Ext(\@Q, \@K)$ are representable by vector bundles over $\#A^1_K.$ 
       \item The morphism $f_*$ is representable by surjective morphism of vector bundles over $\#A^1_K.$
   \item   Any two $K$-points of $\@M$ are naively $\#A^1$-chain connected. 
   \end{enumerate}
\end{proposition}
\begin{proof}
    \noindent\begin{enumerate}
        \item $\@Ext(\@Q, \~{\@K})$ is given by groupoid for each $\#A^1_K$-scheme $S$,  
        \begin{itemize}
            \item the objects are exact sequences $0\to q_S^*(\~{\@K}) \to \~{\@E}\to q_S^*(\@Q) \to 0$ of vector bundles on $C\x_k S$ where $q_S: C\x_k S\to C\x \#A^1_K.$ 
            \item the morphisms are morphisms of exact sequences. 
        \end{itemize}
       Let $p: C\x \#A^1_K\to \#A^1_K$ be the projection morphism and $p_S: C\x_k S\to C$ the base change and $q_S: C\x S\to C\x\#A^1_K$ induced morphism. The moduli stack $\@Ext(\@Q, \~{\@K})$ is given by $$S\mapsto  H^0(S, {\bf Ext^1}_{p_S}(q_S^*\@Q,q_S^*\~{\@K}))$$  By \cite[Proposition 3.1]{Lan83}, this functor is representable by a vector bundle over $\#A^1_K.$ This vector bundle is trivial by Quillen-Suslin theorem (\cite{Qui76}, \cite{Sus76}), hence represented by an affine space. 

       Similarly, the moduli stack $\@Ext(\@Q, \@K)$ is representable by a vector bundle over $\#A^1_K.$

  \item By Yoneda's lemma, the morphism $f_*: \@Ext(\@Q, \~{\@K})\to \@Ext(\@Q, \@K)$ is induced by a morphism of vector bundles over $\#A^1_K.$ Since the morphism of stacks is surjective by~\cite[Proof of Lemma 3.8]{BH08}, the map of vector bundles is surjective. 

  \item By (2), since the diagram is cartesian commutative, the morphism $\gamma$ is representable by the structure morphism of a vector bundle over  $\#A^1_K$, which is trivial hence isomorphic to an affine space. Let $e_1, e_2$ be two $K$-points of $\@M$. The points correspond to $K$-points of an affine space over $K,$ hence there is a naive $\#A^1$-homotopy from   $e_1$ to $ e_2$. 
  \end{enumerate}\end{proof}
We are now all set to prove the main result of the section. Recall the auxilliary stacks from the diagram~\eqref{eq:aux-stacks}.

\begin{theorem}\label{thm:main-text}
   Let $C$ be a geometrically irreducible smooth projective curve over a field $k$ of genus $g\geq 2$ and $C(k)\neq \emptyset$. Let $L$ and $I$ be line bundles on $C$ and $n\geq 1$. Then 
   \begin{enumerate}
       \item The moduli stack $\@Ext(\@F^{univ}, I, 2n-2)$ is naively $\#A^1$-chain connected, hence $\#A^1$-connected. 
       \item The stack $\@Sub_{1, I, 2n}(\@E^{univ})$ is naively $\#A^1$-chain connected, hence $\#A^1$-connected. 
       \item The moduli stack $Sympl_{2n, L}$ of rank $2n$ symplectic vector bundles  on $C$ is naively $\#A^1$-chain connected for each $n\geq 1$, hence $\#A^1$-connected.

   \end{enumerate}
\end{theorem}
\begin{proof}
In view of~\Cref{lem:a1-chain-a1-conn}, it is enough to show that the corresponding stacks are $\#A^1$-chain connected. 

    We show for each $n\geq 1,$ (1) implies (2) and (2) implies (3). Then we show (1) for each $n$ by induction on $n.$ 
\medskip\\
\noindent\_{\bf Step 1}: We show (1) implies (2). 
Let $K$ be a finitely generated field extension of $k$ and let  $[E_1], [E_2]$ be $K$-points of the moduli stack $\@Sub_{1, I, 2n}(\@E^{univ})$. Consider the morphism $$\Psi_I: \@Sub_{1, I, 2n}(\@E^{univ})\to \@Ext(\@F^{univ}, I, 2n-2)$$ and 
 for $i=1, 2$ the images $\Psi_I([E_i])$ of the corresponding $K$-valued points of $\@Ext(\@F^{univ}, I, 2n-2)$. 

Assume $\@Ext(\@F^{univ}, I, 2n-2)$ is naively $\#A^1$-chain connected, hence we can assume WLOG that there is $$h: \#A^1_K\to \@Ext(\@F^{univ}, I, 2n-2)$$ such that $h(i): \Spec K\to   \@Ext(\@F^{univ}, I, 2n-2)$ corresponds to $\Psi_I([E_i])$ for $i=1, 2.$ 
    Let $h$ correspond the classifying morphism $c_{\@H}:  \#A^1_K \to \@Ext(\@F^{univ}, I, 2n-2)$ corresponding to an object $$(\@F, \ c: \@F\otimes \@F\to \@L, \ 0\to \@I\to \@H\to \@F\to 0)$$ where $\@F$ is a symplectic bundle and also $\@L$, $\@I$ are line bundles on $C\x_k \#A^1_K$ and $$0\to \@I\to \@H\to \@F\to 0$$ is an extension of vector bundles on $C\x_k \#A^1_K$.  
    Consider the corresponding extension $$0\to \@F^{dual}\otimes \@L \to \@H^{dual}\otimes\@L\to \@I^{dual}\otimes\@L \to 0.$$ 
    Consider $\@M:=\@Ext(\@I^{dual}\otimes \@L, \@H\xrightarrow{c^{\sharp}\circ p} \@F^{dual}\otimes \@L)/\@H^{dual}\otimes\@L$, the moduli stack over $\#A^1_K$ where for any $\#A^1_K$-scheme $S$ the corresponding groupoid $\@M(S)$ consists of objects         $$ \xymatrix{
    0 \ar[r] & \pi^*(\@H) \ar[d]^{\pi^*(c^{\sharp}\circ p)}\ar[r] \ar[d] & \~{\@E} \ar[r] \ar[d] & \pi^*(\@I^{dual}\otimes\@L) \ar[r] \ar@{=}[d] & 0 \\
    0\ar[r] & \pi^*(\@F^{dual}\otimes \@L) \ar[r] & \pi^*(\@H^{dual}\otimes\@L) \ar[r] & \pi^*(\@I^{dual}\otimes\@L) \ar[r] & 0 
    }$$ where  $\pi: C\x_k S\to C\x_k \#A^1_K$ and, the top row is an exact sequence of vector bundles  over $C\x_k S$. 
    
    By~\cite[Proof of Corollary 3.9]{BH08}, there is a cartesian square $$\xymatrix{
    \@M \ar[r] \ar[d] &  \@Sub_{1, I, 2n}(\@E^{univ}) \ar[d]^{\Psi_I} \\
    \#A^1_K \ar[r]^-{c_{\@H}}  & \@Ext(\@F^{univ}, I, 2n-2) 
    }$$ of moduli stacks. 
   Lifting the points $[E_i]$ of the moduli stack $\@Sub_{1, I, 2n}(\@E^{univ})$ to $K$-points $e_i$ of the moduli stack $\@M$ respectively, where $e_1$ lies over $0: \Spec K\to \#A^1_K$ and $1: \Spec K\to \#A^1_K$. Now it is enough to show that the $K$-points $e_1, e_2$ of $\@M$ are naively $\#A^1$-chain connected.  

   By \cite[Definition 3.6, (6)]{BH08}, there is a cartesian square of canonical 1-morphisms $$\xymatrix{
   \@M \ar[r]^-{\phi} \ar[d] & \@Ext(\@I^{dual}\otimes\@L, \@H)  \ar[d]^{(c^{\sharp}\circ p)_*} \\ 
   \#A^1_K\ar[r] &  \@Ext(\@I^{dual}\otimes\@L,\@F^{dual}\otimes \@L)
   }$$ 
   where the right vertical morphism is a morphism of moduli stacks over $\#A^1_K.$ 
Now we are done by taking $\~{\@K}=\@H, \@K=\@F^{dual}\otimes \@L,$ and $ \@Q=\@I^{dual}\otimes\@L$ in~\Cref{prop:ext}.   
\\
\medskip
\noindent   \_{\bf Step 2}: We show (2) implies (3).

 Let $K$ be a finitely generated field extension over $k$  and $[E_1], [E_2]$ be two $K$-valued points of the moduli stack $Sympl_{2n, L}$ corresponding to symplectic bundles $E_1$ and $E_2$, respectively, on $C_K=C\x_k\Spec K$. By choosing a line bundle $I$ of deg $\ll0$, we can assume $I$ is a sub-bundle of both $E_1$ and $E_2.$ Hence $[E_1]$ and $[E_2]$ are in the image of the morphism of moduli stacks $$\@Sub_{1, I, 2n}(\@E^{univ})\to Sympl_{2n, L}$$ again denoted by $[E_i]=(E_i, b_i: E_i\otimes E_i\to L, I\subset E_i)$ for $i=1, 2$.   Now by (2) we are done.
\medskip\\
 \noindent    \_{\bf Step 3}: We show (1) by induction on $n\geq 1$. 

 For $n=1$ we have the isomorphisms, $Sympl_{0, L}\simeq \Spec k$ and $\@Ext(\@F^{univ}, I, 0)\simeq B\#G_m$ which is naively $\#A^1$-chain connected. 
 
 Assume $n>1$ and by induction hypothesis for $n-1$, $\@Ext(\@F^{univ}, I, 2n-4)$ is naively $\#A^1$-chain connected.  
  By the induction hypothesis, and by (2) implies (3), we have that $Sympl_{2n-2, L}$ is naively $\#A^1$-chain connected. Consider the morphism $$\Phi'_I: \@Ext(\@F^{univ}, I, 2n-2)\to Sympl_{2n-2, L}.$$ 
Let $[\@H_i]=(\@F_i, c_i: \@F_i\otimes \@F_i\to \@L, 0\to\@I_i \to \@H_i \to \@F_i\to 0)$ for $i=1, 2$ be two objects of the groupoid $\@Ext(\@F^{univ}, I, 2n-2)(\Spec K)$. Let $$[\@I_i\oplus \@F_i]:=(\@F_i, c_i: \@F_i\otimes \@F_i\to \@L, 0\to\@I_i \to \@I_i\oplus \@F_i\to \@F_i\to 0)$$ for each $i=1, 2$. 
Then for each fixed $i$, there is an inclusion of moduli functors $$\@Ext(\@F_i, \@I_i) \to \@Ext(\@F^{univ}, I, 2n-2)$$ so that the image contains the $K$-points $[\@H_i]$ and $[\@I_i\oplus \@F_i]$. 
By \cite[Proposition 3.1]{Lan83}, the moduli stack $\@Ext(\@F_i, \@I_i)$ is represented by a vector bundle, hence naively $\#A^1$-chain connected and in particular, $[\@H_i]$ and $[\@I_i\oplus \@F_i] $ are naively $\#A^1$-chain connected in $\@Ext(\@F^{univ}, I, 2n-2).$

Consider the morphism of moduli stacks $$s_I:  Sympl_{2n-2, L}\to \@Ext(\@F^{univ}, I, 2n-2)$$ which is defined on groupoids for each $k$-scheme $S$ as follows $$(\@F, c:\@F\otimes \@F\to \@L)\mapsto (\@F, c:\@F\otimes \@F\to \@L, 0\to \pi^*I\to \pi^*I\oplus\@F\to\@F \to 0)$$  where $\pi: C\x S\to C$ is the projection morphism. The morphism $s_I$ is a section of $\Phi'_I.$ 

Clearly the $K$-valued points $[\@I_i\oplus \@F_i]$ are in the image of the morphism $s_I$ for $i=1, 2$ and since as noted before the moduli stack $Sympl_{2n-2, L}$ is naively $\#A^1$-chain connected, the points $[\@I_1\oplus \@F_1]$ and $[\@I_2\oplus \@F_2]$ in the groupoid $\@Ext(\@F^{univ}, I, 2n-2)(\Spec K)$ are naively $\#A^1$-chain homotopic. Hence we are done. 
\end{proof}

    \section{Applications}\label{sec:appln}
    
\subsection{\texorpdfstring{$\#A^1$-h-cobordance}{}} In this subsection, we record a simple observation.  
\begin{definition}(\cite[Definition 3.1.1]{AM11})
    We say two smooth proper $k$-schemes $X_0, X_1$ are directly $\#A^1$-h-cobordant if there is a smooth $k$-scheme with a morphism $f: X\to \#A^1$ such that 
    \begin{enumerate}
        \item $f^{-1}(0)$ is isomorphic to $X_0$ and $f^{-1}(1)$ is isomorphic to $X_1$ and
        \item the natural inclusions $X_i\to X$ are $\#A^1$-equivalences for $i=0,1$.
    \end{enumerate}

    We say $X_0, X_1$ are $\#A^1$-h-cobordant if they are equivalent by the relation generated by the relation of direct $\#A^1$-h-cobordance.
\end{definition}

As an easy consequence of \cite[Theorem 1.2]{HY24} we record the following proposition.
\begin{proposition}
   Let $C$ be a smooth projective curve over an infinite field $k$. Let $(\@E, b: \@E\otimes \@E\to \@L_1)$ and $(\@F, c: \@F\otimes \@F\to \@L_2)$ be symplectic vector bundles of rank $2n$ on $C$. Let $X=\#P(\@E)$ and $Y=\#P(\@F)$ be the associated projective bundles. Consider the following statements.
\begin{enumerate}
    \item $X$ and $Y$ are $\#A^1$-equivalent.
        \item $X$ and Y are $\#A^1$-h-cobordant. 
        \item  $\det \@E \otimes (\det\@F)^{-1}=\@L^{\otimes 2n}$ for some $\@L\in \Pic(C).$
\item[(3')] $\@L_1 \otimes \@L_2^{-1} = \@N^{\otimes 2}$ for some $\@N \in \Pic(C)$.
     \end{enumerate}

     Then \begin{enumerate}[label=(\alph*)]
     \item $(1)\iff (2) \iff $(3). 
            \item $(3')\Rightarrow (1)$
    \item  if $n\geq 1$ is odd, then $(3) \iff (3')$. 
     
     \end{enumerate}
\end{proposition}

\begin{proof} The proof follows from \cite[Theorem 1.2]{HY24} by simply observing that the condition (3) in \cite[Theorem 1.2]{HY24} translates into $\@L_1^{\otimes n}\otimes  \@L_2^{\otimes -n}=\@L^{\otimes 2n}$ in $\Pic (C)$ for some $\@L\in \Pic(C)$. In other words,  $\@L_1\otimes  \@L_2^{\otimes -1}\otimes \@L^{\otimes -n}$ is of order $n$ in $\Pic(C).$ But $n$ is odd, hence $\@L_1\otimes  \@L_2^{\otimes -1}\otimes \@L^{\otimes -2}\simeq \@L'^{\otimes 2}$ for some $\@L'\in\Pic (C).$ Hence $\@L_1\otimes  \@L_2^{\otimes -1}\simeq \@N^{\otimes 2}$ for $\@N=\@L\otimes \@L'.$ 
 \end{proof}

\subsection{\texorpdfstring{$\#A^1$-connectedness of Moduli stack of quasi-parabolic symplectic bundles}{}}\label{sec:quasi-pb-symplectic}
Fix a geometrically irreducible smooth projective curve $C$ of genus $\geq 2$ over a field $k$ and a finite subset $D \subset C$ of distinct $k$-rational points. With extra care, one may be able to generalize this to the case where $D$ consists of closed points rather than $k$-rational points. But for simplicity, we only consider the above setup.    

We define the quasi-parabolic bundle which is a reformulation of principal bundle with a flagging of certain type as defined in~\cite[Section 3.2]{HS10}.    

\begin{definition}\label{def:quasi-para-data} 
    A quasi-parabolic data $(e, \_m)$ on $C$ of rank $2n$ along $D$ is defined as follows: 
\begin{itemize}
    \item For each point $p\in D$, $e(p)\in \#N$ and a tuple of postive integers $\_m(p)=(m_{p, 1}, \dots, m_{p, e(p)})$ such that $m_{p, 1}+ \dots+ m_{p, e(p)}=2n.$
\item The set $e=\{e(p) : p\in D\}$ is the system of lengths. 
\item The set $\_m=\{\_m(p) : p\in D\}$ is the system of multiplicities.
\end{itemize} 
\end{definition}

\begin{definition}[quasi-parabolic bundles]\label{def:q-parabolic}

Fix a line bundle $L$ on $C$ and a quasi-parabolic data $(e, \_m)$ on $C$ of rank $2n$ along $D$. 

For any $k$-scheme $S$ and the projection $\pi: C\x_k S\to C$, a quasi-parabolic bundle of rank $2n$ parametrized by $S$ of quasi-parabolic data $(e, \_m)$ denoted by $\@E_*$  given by
\begin{itemize}
    \item a vector bundle $\@E$ on $C\x S$
   \item for each $p\in D$, there is a flag of vector bundles of the pull back vector bundle $\@E_p=\@E|_{\{p\}\x S}$ on $\{p\}\x S$ 
$$\@E_p= \@E_{p, 1}\supsetneq \cdots \supsetneq \@E_{p, e(p)}\supsetneq  \@E_{p, e(p)+1}=\{0\}$$
such that 
\begin{itemize}
     
 \item $e(p)$ is the length of the flag and 
    \item  $m_{p,i}=\rank~\@E_{p,i}-\rank~\@E_{p,i+1}$ are called multiplicities at p.
\end{itemize}
\end{itemize}

     We recall the following definition. We note that usually the following definition is considered when the base field is complex numbers. 
\end{definition}
\begin{definition}[quasi-parabolic symplectic bundles]\label{def:q-parabolic-sympl} Fix a line bundle $L$ on $C$ and a quasi-parabolic data $(e, \_m)$ on $C$ of rank $2n$ along $D$. 

For any $k$-scheme $S$ and the projection $\pi: C\x_k S\to C$, a quasi-Parabolic symplectic bundle of rank $2n$ parametrized by $S$ of quasi-parabolic data $(e, \_m)$ is given by tuple $(\@E_*, \phi_*, \@L(D))$ where 
\begin{itemize}
    \item $\@E_*$ is a quasi-parabolic bundle parametrized by $S$ in the sense of \Cref{def:q-parabolic}.
    \item $\@L$ a line bundle on $C\x_k S$, Zariski locally on $S$ isomorphic to $\pi^*L$ where $\pi:C\x_k S\to C$ is the projection.
    
    $\@L(D)$ denotes the quasi-parabolic bundle parametrized by $S$, 
    given by the line bundle $\@L$ with the special quasi-parabolic structure, \ie $e(p)=1$ for all $p\in D.$ 
    \item an anti-symmetric morphism of parabolic bundles $$\phi_*: \@E_*\otimes \@E_*\to \@L(D)$$
    such that the induced morphism 
    $$\~{\phi}: \@E_*= \@E_* \otimes \@O_{C\x S} \to  \@E_*\otimes (\@E_* \otimes \@E^{\vee}_*) = (\@E_* \otimes \@E_*) \otimes \@E^{\vee}_*\xrightarrow{\phi\otimes {\rm Id}}  \@L(D) \otimes \@E^{\vee}_*$$
    is an isomorphism.
\end{itemize}
\end{definition}
\begin{definition}\label{def:moduli-q-para-sympl}
    Fix a line bundle $L$ on $C$ and a quasi-parabolic data $(e, \_m)$ on $C$ of rank $2n$ along $D$. One endows $L(D)$ with a parabolic bundle structure where for each point $p\in D,$ with a filtration $\@L_p\supset 0$ as in \cite[Remark 2.2]{BCD26}. 
    
One defines the Moduli stack $qPSympl^{e, \_m}_{2n, L(D)}$ of quasi-Parabolic symplectic bundles of rank $2n$ on $C$ as follows. 
More precisely, on each $k$-scheme $S$, the groupoid is given by 
\begin{itemize}
    \item the objects are $(\@E_*, \phi_*, \@L(D))$ where $$\phi_*: \@E_*\otimes \@E_*\to \@L(D)$$ where $\@E_*$ is a symplectic quasi-parabolic bundle on $C\x_k S$ in the sense of \Cref{def:q-parabolic-sympl}, \cite[Definition 2.4]{BCD26} and $\@L$ a line bundle on $C\x_k S$, Zariski locally on $S$ isomorphic to $\pi^*L$ where $\pi:C\x_k S\to C$ is the projection. 
    \item the morphism  $(\@E_*, \phi_*, \@L(D))\to (\@E'_*, \phi'_*, \@L'(D))$ is an isomorphism of underlying parabolic bundles and isomorphism of parabolic line bundles $\@L(D) \to \@L'(D)$ compatible with the anti-symmetric bilinear forms $\phi, \phi'.$ 
\end{itemize}

\end{definition}
\begin{theorem}\label{thm:quasi-parabolic-sympl}
        Let $C$ be a geometrically irreducible smooth projective curve of genus $g\geq 2$ over a field $k$ and a finite subset $D \subset C$ of distinct $k$-rational points. Fix a quasi-parabolic data $(e, \_m)$ on $C$ of rank $2n$ along $D$. 
    
    Then $qPSympl^{e, \_m}_{2n, L}$ is $\#A^1$-connected. 
\end{theorem}
\begin{proof}
     We note the following fiber product diagram of stacks $$\xymatrix{
    qPSympl^{e, \_m}_{2n, L}\ar[d] \ar[r]& qPSympl^{e, \_m}_{2n} \ar[d] \\
    Sympl_{2n, L} \ar[r] & Sympl_{2n}}$$
where the vertical maps are forgetful morphisms (forgetting the quasi-parabolic structure). To see that the right vertical map is well defined, the proof in~\cite[Lemma 3.1]{BCD26} works in our setting. Recall that the right vertical map is given by $$(\@E_*, \phi_*: \@E_*\otimes \@E_*\to \@L(D))\mapsto  (\@E, \phi:\@E\otimes \@E\to \@L).$$ It follows that the above diagram is cartesian. 

    By~\cite[Lemma 3.2.2(ii)]{HS10} the right vertical map is Zariski locally trivial with fiber a product of finite flag varieties, hence the same holds for the left vertical map.  Now each Flag variety is $\#A^1$-chain connected and by~\Cref{thm:main}(3), $Sympl_{2n, L}$ is $\#A^1$-connected, hence the moduli stack $qPSympl^{e, \_m}_{2n, L}$ is $\#A^1$-connected. 
\end{proof}

\end{document}